\documentclass[11pt]{amsart}
\usepackage{amssymb}
\usepackage{verbatim}
\usepackage{color}

\newtheorem{theorem}{Theorem}[section]          
\newtheorem{lemma}[theorem]{Lemma}             
\newtheorem{corollary}[theorem]{Corollary}         

\newcommand{\bN}{\mathbf{N}}          
\newcommand{\bL}{\mathbf{L}}          
\newcommand{\bW}{\mathbf{W}}    
\newcommand{\bS}{\mathbf{S}}    
    
\newcommand{\bU}{\mathbf{U}}    

\newcommand{\bw}{\alpha}    
\newcommand{\bv}{\beta}    

\newcommand{\blue}[1] {\textcolor{blue}{#1}}
\newcommand{\red}[1] {\textcolor{red}{#1}}

\newcommand{\defo}[1] {\emph{\textcolor{blue}{#1}}}

\newcommand{\lyn}{\mathrm{lyn}}

\begin{document}

\title[A Pascal-like Bound for the Number of Necklaces]{A Pascal-like
Bound for the Number of Necklaces with Fixed Density}
\author{I.~Heckenberger}
\address{Philipps Universit\"at Marburg, Fachbereich Mathematik und Informatik,
Hans-Meerwein-Str., 35032 Marburg, GERMANY}
\email{heckenberger@mathematik.uni-marburg.de}
\author{J.~Sawada}
\address{School of Computer Science, University of Guelph, Canada}
\email{jsawada@uoguelph.ca}

\maketitle

\begin{abstract}
A bound resembling Pascal's identity is presented for binary necklaces with fixed density using Lyndon words with fixed density.   The result is generalized
to $k$-ary necklaces and Lyndon words with fixed content.
The bound arises in the study of Nichols algebras of diagonal type.
\end{abstract}

\section{Introduction}

A \defo{necklace} is the lexicographically smallest word in an equivalence class of words under rotation.  A \defo{Lyndon word} is a necklace that is strictly smaller than any of its non-trivial rotations.
The \defo{density} of a binary word is the number of 1s it contains.  Let $\bN(n,d)$ denote the set of all binary necklaces of length $n$ and density $d$.  Similarly, let  $\bL(n,d)$ denote the set of all binary Lyndon words of length $n$ and density $d$.  Let the cardinality of these two sets be denoted by $N(n,d)$ and $L(n,d)$, respectively.  
The following formulae are well-known for any $n \geq 1$ and $0 \leq d \leq n$  (see \cite{MR0142480} and \cite[Sect.\,2]{MR1744685}):
\begin{align} \label{eq:N2L2}
  N(n,d)&=\frac{1}{n} \sum_{j\mid \gcd(n,d)}\phi(j){n/j \choose d/j},&
  L(n,d)&=\frac{1}{n} \sum_{j\mid \gcd(n,d)}\mu(j){n/j \choose d/j},  \nonumber
\end{align}
where $\phi $ and $\mu $ denote Euler's totient function and the M\"obius
function, respectively.
Note that $N(n,d) = N(n,n{-}d)$ and $L(n,d) = L(n,n{-}d)$.
Our main result is to prove the following Pascal's identity-like bound on 
$N(n,d)$.

\begin{theorem} \label{th:Pascal}
For any $0 <  d < n$, 
\[ N(n,d)  \leq L(n{-}1,d) + L(n{-}1,d{-}1). \]
\end{theorem}

Bounds on necklaces and Lyndon words like the one presented in the above theorem are generally difficult to prove directly from their enumeration formulae.  Previous bounds on these objects used in algorithm analysis~\cite{MR1744685} use a  combinatorial-style proof, and that is the approach we follow in this paper.
To prove this theorem, we actually show something stronger.  Let $\Sigma = \{0,1,2, \ldots , k-1 \}$ denote an alphabet of size $k$.
Let $\bW_k(n_0,n_1, \ldots, n_{k-1})$ denote the set of all words over $\Sigma$ where each symbol $i$ appears precisely $n_i$ times.  
Such a set is said to be a set with \defo{fixed content}, as used by
\cite[Sect.\,18.3.3]{Arndt}.  
In a similar manner let $\bN_k(n_0,n_1, \ldots, n_{k-1})$ denote the set of necklaces with the given fixed content and
let $\bL_k(n_0,n_1, \ldots, n_{k-1})$ denote the set of Lyndon words with the given fixed content.  
Let the cardinality of these two sets be denoted by $N_k(n_0,n_1, \ldots, n_{k-1})$ and $L_k(n_0,n_1, \ldots, n_{k-1})$, respectively.
In \cite{MR0142480} and \cite[Sect.\,2]{MR1744685}, explicit formulas for the number of necklaces
and Lyndon words with fixed content are given:
\begin{align}
  N_k(n_0,n_1,\dots,n_{k-1})&=
\frac 1n\sum_{j\mid \gcd (n_0,\dots,n_{k-1})}\phi(j)
\frac{(n/j)!}{(n_0/j)!\cdots (n_{k-1}/j)!},  \nonumber  \\
  L_k(n_0,n_1,\dots,n_{k-1})&=
\frac 1n\sum_{j\mid \gcd (n_0,\dots,n_{k-1})}\mu(j)
\frac{(n/j)!}{(n_0/j)!\cdots (n_{k-1}/j)!},   \nonumber
\end{align}
where $n=n_0+n_1+\cdots +n_{k-1}$.  In Section~\ref{sec:proofs}, we prove the following more general result.

\begin{theorem} \label{th:ineq}
 Let $k\ge 2$ and $n_0,n_1,\dots,n_{k-1}\ge 1$ be positive integers. Then
\begin{align} \label{eq:ineq}
  N_k(n_0,\dots,n_{k-1})\le \sum_{i=0}^{k-1} L_k(n_0,\ldots,n_{i-1},
  n_i-1,n_{i+1},\ldots,n_{k-1}).  \nonumber
\end{align}
Moreover, the inequality is strict if $k > 2$.
\end{theorem}
Observe that when $k=2$,  Theorem ~\ref{th:ineq} simplifies to Theorem~\ref{th:Pascal}.  After presenting some preliminary materials in Section~\ref{sec:back}, we prove the theorems in
Section~\ref{sec:proofs}.

\subsection{An Application}

	The inequalities in Theorems~\ref{th:Pascal} and \ref{th:ineq}
	seem to be new. They arised with the study
	of Nichols algebras of diagonal type in \cite{HZ}
	in order to determine whether such a
	Nichols algebra is a free algebra. Roughly, the inequality implies that a
	certain rational function is in fact a polynomial, and freeness of the Nichols
	algebra holds if none of these polynomials vanish at the point
	of an affine space determined by the braiding of the Nichols algebra.
        The calculation of the zeros of such a polynomial simplifies significantly if
	the inequality is known to be strict.  Strictness when $k=2$ is further discussed in Section~\ref{sec:special}.

\section{Background}  \label{sec:back}

A word is called a \defo{prenecklace}, if it is the prefix of some necklace.
For any non-empty word $\bw$ let $\lyn(\bw)$ be the length of the longest prefix
of $\bw$ that is a Lyndon word.
\begin{theorem} (Fundamental theorem of necklaces) \cite[Thm.\,2.1]{MR1788836} \label{thm:fund}
	\label{th:ftn}
  Let $n\ge 2$,
	$k\ge 2$, and let $\bw =a_1\cdots a_{n-1}$ be a prenecklace over the alphabet $\Sigma = \{0,1, \ldots , k-1\}$.
  Let $p=\lyn(\bw)$ and let $b \in \Sigma$. Then $\bw b$ is a prenecklace if and only if $a_{n-p}\le b < k$.
  In this case,
  $$ \lyn(\bw b)=\begin{cases}
    p & \text{if $b=a_{n-p}$,}\\
    n & \text{if $a_{n-p}<b<k$.}
  \end{cases}
  $$
\end{theorem}
The following corollaries follows immediately from the previous theorem.
\begin{corollary} \label{cor:lyn}
If $\bw = a_1a_2\cdots a_n$ is a prenecklace and  $b > a_n$, then the word
$a_1a_2\cdots a_{n-1}b$ is a Lyndon word.
\end{corollary}
\begin{corollary} \label{cor:neck}
If $\bw = a_1a_2\cdots a_n$ is a necklace and $b > a_1$, then $\bw b$ is a Lyndon word.
\end{corollary}
The following result also follows from the above theorem and corresponds to
~\cite[Lem.\,2.3]{MR1788836}.
\begin{corollary} \label{cor:pre}
If $\bw = a_1a_2\cdots a_n$ is a prenecklace and $p = \lyn(\bw)<n$,
then $\bw = (a_1a_2\cdots a_p)^j a_1a_2\cdots a_i$
for some $j\geq 1$ and $1 \leq i \leq p$.   
\end{corollary}
%
\section{Proof of Main Theorems} \label{sec:proofs}

A necklace $\bw = a_1a_2\cdots a_n$ is said to be \defo{stable} if $a_1a_2\cdots a_{n{-}1}$ is a Lyndon word; otherwise $\bw$ is \defo{unstable}.    
We prove Theorem~\ref{th:ineq}, which generalizes Theorem~\ref{th:Pascal}, by partitioning $\bN_k(n_0,n_1, \ldots, n_{k-1})$ into two 
sets $\bS$ and $\bU$, which contain the stable and unstable necklaces of  $\bN_k(n_0,n_1, \ldots, n_{k-1})$, respectively.

\begin{lemma} \label{lem:S}
 Let $k\ge 2$ and $n_0,n_1,\dots,n_{k-1}\ge 1$ be positive integers. Then
\[   |\bS|  = \sum_{i=1}^{k-1} L_k(n_0,\ldots,n_{i-1}, \blue{n_i{-}1},n_{i+1},\ldots,n_{k-1}). \]
\end{lemma}
\begin{proof}
Since each $n_i > 0$, every necklace in $\bS$  must begin with 0 and end with a non-0.  By further partitioning $\bS$ by
its last symbol, the result follows from Corollary~\ref{cor:neck} and the definition of stable.
\end{proof}

It remains to show that $|\bU| \leq L_k(\blue{n_0{-}1},n_1, n_2, \ldots ,n_{k-1})$.   We assume $k \geq 2$ and each $n_i \geq 1$.
Let $\bw = a_1a_2\cdots a_n$ be a necklace in $\bU$.  Let $\bw' = a_1a_2\cdots a_{n-1}$ and let $x = a_n$.  Since $\bw$ is unstable, 
$\bw'$ is a prenecklace, but not a Lyndon word.  Thus, applying Corollary~\ref{cor:pre}, we can write $\bw'$ as
$(a_1a_2\cdots a_p)^j a_1a_2\cdots a_i$ where $j \geq 1$ and $1 \leq i \leq p$.   Let $z$ be the largest index less than or equal to $i$ such that $a_z = 0$.  
Thus 
\[ \bw = (a_1a_2\cdots a_p)^j a_1a_2\cdots a_{z-1} \red{a_z} a_{z+1} \blue{a_{z+2}\cdots a_i x}.\] 
Consider the   function $f:  \bU \rightarrow \bL_k(n_0{-}1,n_1, n_2, \ldots ,n_{k-1})$ as follows:
\begin{center}
$f(\bw) = \left\{ \begin{array}{ll}
        (a_1a_2\cdots a_p)^j  a_1a_2\cdots a_{z-1} \blue{x}\  &\ \  \mbox{if  $z=i$;} \\
       (a_1a_2\cdots a_p)^j   \blue{x a_i a_{i-1} \cdots a_{z+2}} a_1a_2\cdots a_{z-1} a_{z+1}  &\ \  \mbox{if  $z<i$.}
         \end{array} \right.$
\end{center}
Clearly $f(\bw)$  has the required content.  
To see that $f(\bw)$ is a Lyndon word, observe first that $(a_1a_2\cdots a_p)^j$
is a necklace beginning with 0.  Since each symbol in
$\blue{x a_i a_{i-1} \cdots a_{z+2}}$ is non-0,
$\beta=(a_1a_2\cdots a_p)^j   \blue{x a_i a_{i-1} \cdots a_{z+2}}$
is a Lyndon word by Corollary~\ref{cor:neck}.
Both words $(a_1a_2\cdots a_p)^j  a_1a_2\cdots \red{a_{z}}$ and 
$(a_1a_2\cdots a_p)^j \blue{x a_i a_{i-1} \cdots a_{z+2}} a_1a_2\cdots
\red{a_{z}}$  are prenecklaces, since they are beginnings of
the necklaces $(a_1\cdots a_p)^{j+2}$ and $\beta^2$, respectively.
Thus since  $x, a_{z+1} > 0$ and $a_z = 0$ Corollary~\ref{cor:lyn}  implies $f(\bw)$ is a Lyndon word.

\begin{lemma}  \label{lem:U}
 Let $k\ge 2$ and $n_0,n_1,\dots,n_{k-1}$ be positive integers. Then
\[  |\bU| \leq L_k(\blue{n_0{-}1},n_1, n_2, \ldots ,n_{k-1}). \]
\end{lemma}
\begin{proof}
We prove that $f$ is one-to-one.  Consider two necklaces $\bw=a_1a_2\cdots a_n$ and $\bv = b_1b_2\cdots b_n$ in $\bU$.
As discussed when defining $f$, we can write $\bw$ and $\bv$ as:
\[ \bw = (a_1a_2\cdots a_p)^j a_1a_2\cdots a_{z-1} \red{a_z} a_{z+1} \blue{a_{z+2}\cdots a_i x},\] 
\[ \bv = (b_1b_2\cdots b_{p'})^{j'} b_1b_2\cdots b_{z'-1} \red{b_{z'}} b_{z'+1} \blue{b_{z'+2}\cdots b_{i'} x'}.\]
%
Suppose $f(\bw) = f(\bv)$.   If $jp = j'p'$, then clearly $p=p'$ and $\bw = \bv$.  Otherwise, 
without loss of generality assume that $jp < j'p'$.  Then $p'>jp$ and
$a_1a_2\cdots a_p = b_1b_2\cdots b_p$. 
Consider two cases based on $\ell = |a_{z+2} \cdots a_ix|$.  
\begin{itemize}
\item Case  $jp + \ell \geq j'p'$.  This implies $\ell > 0$. Let  $\ell' =
	|b_{z'+1} \cdots b_{i'}x'|$.  Recall $a_1 = b_1 = 0$ and each letter of
	$a_{z+1} \cdots a_ix$ and $b_{z'+1} \cdots b_{i'}x'$ is non-0.  Thus, since
	$f(\bw) = f(\bv)$,
we must have $jp + \ell  =  j'p' + \ell'$. But this means  $x' = a_{i'+1} = b_{i'+1}$, which means that $\bv$ ends with $b_1b_2\cdots b_{i'+1}$ which is a proper prefix of $b_1b_2\cdots b_{p'}$
contradicting $\bv$ being a necklace.
\item Case  $jp + \ell < j'p'$.  In this case,   for $f(\bw) = f(\bv)$, it must be that some suffix of $(b_1b_2\cdots b_{p'})^{j'} $ is a prefix of $a_1a_2\cdots a_{z-1} = b_1b_2\cdots b_{z-1}$ which contradicts $b_1b_2\cdots b_{p'}$ being a Lyndon word.
\end{itemize}

\end{proof}

%
Together Lemma~\ref{lem:S} and Lemma~\ref{lem:U} prove Theorem~\ref{th:Pascal}. To complete the proof of Theorem~\ref{th:ineq}, 
the following lemma proves that for $k>2$ the function $f$ is not a bijection.

\begin{lemma}
The function $f$ is not a surjection when $k  > 2$.
\end{lemma}
\begin{proof}
First we consider the special cases of $n_0 = 1$ or $n_0 = 2$.  Then we consider the parity of $n_0$.  For the latter three cases we demonstrate a word
in $\bL_k(n_0{-1}, n_1, n_2,\ldots,n_{k-1})$ that is not in the range of $f$.
\begin{itemize}
\item Case: $n_0 = 1$.   $\bU$ is empty, but $\bL_k(n_0{-1}, n_1, n_2,\ldots,n_{k-1})$ is not. 
\item Case: $n_0 = 2$.   There is no necklace in $\bU$ that maps to
	the word $\alpha=0 (k{-}1)^{n_{k-1}} \cdots  2^{n_2} 1^{n_1}$. Indeed,
any necklace $\gamma $ with $f(\gamma)=\alpha $
has to start with $0(k{-}1)$, since $\alpha $ starts with
$0(k-1)$. Moreover, $\gamma $ has the subword $01$ since $\alpha $ ends with $1$.
\item Case:  $n_0 = 2j+1$ for $j \geq 1$. There is no necklace in $\bU$ that maps to $0^{j} 1^{n_1}  2^{n_2} \cdots (k{-}1)^{n_{k-1}-1} 0^{j} (k{-}1)$ because
such a necklace would have to start with $0^j1$ but have the subword $0^{j+1}$.
\item Case: $n_0 = 2j$ for $j \geq 2$.  There is no necklace in $\bU$ that maps to $0^{j} (k{-}1)^{n_{k-1}} \cdots  3^{n_3}2^{n_2} 1^{n_1-1} 0^{j-1}1$  because
such a necklace would have to start with $0^{j} (k{-}1)$ but have the subword  $0^j1$.
\end{itemize}
\end{proof}

\section{Special Cases When $N(n,d) = L(n{-}1,d) + L(n{-}1,d{-}1)$} \label{sec:special}

      In this section we discuss when the inequality given by Theorem~\ref{th:Pascal} is equality.   
      
      \begin{lemma} 
      If $d \in \{1,2,n{-}2,n{-}1\}$  and $0 < d < n$ then $N(n,d) = L(n{-}1,d) + L(n{-}1,d{-}1)$ except for $(n,d) = (2,1)$.      
      \end{lemma}
      \begin{proof}
      Recall that $N(n,d) = N(n,n{-}d)$ and $L(n,d) = L(n,n{-}d)$.  Thus it suffices to prove the result for $d=1$ and $d=2$.
      For $d=1$ and $n > 2$, $\bN(n,1) = \{0^{n-1}1\}$,  $\bL(n{-}1,1) = \{0^{n-2}1\}$,  and $\bL(n{-}1,0) = \emptyset$, and thus the result holds.
      For $d=2$, note that $N(n,2)  = \lfloor \frac{n}{2} \rfloor$, $L(n{-}1,2)  = \lfloor \frac{n-2}{2} \rfloor$, and $L(n{-}1,1) = 1$, and thus the result holds.
       \end{proof}

      We now consider the other values of $d$.   
            
      \begin{lemma}
      If $2 <  d < n-2$  then $N(n,d) = L(n{-}1,d) + L(n{-}1,d{-}1)$ if and only if $(n,d) \in \{ (6,3), (7,3), (7,4), (8,4), (9,3), (9,6) \}$. 
      \end{lemma}
      \begin{proof}
      The claim can easily be verified by direct computation for $n \leq 10$.  Recall that $N(n,d) = N(n,n{-}d)$ and $L(n,d) = L(n,n{-}d)$.   For simplicity, let $z = n{-}d$ (the number of 0s).
      We consider $n>10$ and $2<  z \leq n/2$ in three cases for $z=3$, $z=4$, and $z \geq 5$. 
      Each result is proved by specifying a Lyndon word $\bv \in \bL(n{-}1,d) = \bL_2(z{-}1,d)$ not in the range of $f$ when the domain is the set of unstable necklaces in $\bN(n,d) =  \bN_2(z,d)$.

      \begin{itemize}
      
\item Case: $z=3$.  Depending on the parity of $n$ let  $\bv$ be either  $01^a01^{a+1}$  or  $01^a01^{a+2}$.  Since $n > 10$, $a \geq 3$.
Consider any necklace $\bw$ in $\bN_2(3,d)$.  It must be of the form $01^{a_1}01^{a_2}01^{a_3}$ for non-negative integers $a_1,a_2,a_3$ with $a_2, a_3 \geq a_1$.
But for any such (unstable) $\bw$,  $f(\bw) = 01^{a_1}01^{a_2+a_3} \neq \bv$.

\item Case: $z=4$.  Depending on the value of $(n \bmod 3)$ let  $\bv$ be one of
	the words $01^{a}01^a01^{a+1}$, $01^{a}01^{a+1}01^{a+1}$, and
	$01^{a}01^{a+1}01^{a+2}$.  Since $n > 10$, $a \geq 2$.
Consider any necklace $\bw$ in $\bN_2(4,d)$.  It must be of the form
$01^{a_1}01^{a_2}01^{a_3}01^{a_4}$ for non-negative integers $a_1,a_2,a_3,a_4$
with $a_2, a_3, a_4 \geq a_1 \geq 0$. 
But for any such (unstable) $\bw$, it is not difficult to see that either
$a_1=a_3<a_2$, $f(\bw) = 01^{a_1}0^{a_2+a_4-1}01^{a_3+1}$,
or $f(\bw) =  01^{a_1}01^{a_2}01^{a_3+a_4}$.  Neither case is equal to $\bv$.

\item Case:  $z \geq 5$.  Consider $\bv = 0011^{a}01(01)^{b}$.  For any $n> 10$ and $z \leq n/2$, such Lyndon word can be constructed with length $n$ and exactly $z{-}1$ 0s for some $a, b \geq 1$.
But for any necklace $\bw$ in $\bN_2(z,d)$, by the definition of $f$, it is not difficult to observe that $f(\bw) \ne \bv$. 

\end{itemize}
      
      \end{proof}


\end{document}